\documentclass[11pt,letterpaper]{amsart}

\usepackage{amssymb,amsmath,amsthm,amstext,amscd,latexsym,graphics,graphicx,bbm,caption, mathtools}
\usepackage{relsize}
\usepackage[usenames,dvipsnames,svgnames,table]{xcolor}
\usepackage{hyperref}
\usepackage{tikz}
\usetikzlibrary{positioning}
\usepackage{tikz-cd}
\tikzset{>=stealth}
\hypersetup{plainpages=false,colorlinks=true, pagebackref}
\hypersetup{citecolor=Sepia,linkcolor=blue, urlcolor=blue}

\newtheorem{intro-thm}{Theorem}[]
\theoremstyle{plain}
\newtheorem{thm}{Theorem}[section]
\newtheorem{theorem}[thm]{Theorem}

\newtheorem{lemma}[thm]{Lemma}
\newtheorem{corollary}[thm]{Corollary}
\newtheorem{proposition}[thm]{Proposition}

\theoremstyle{definition}
\newtheorem{remark}[thm]{Remark}

\newtheorem{definition}[thm]{Definition}
\newtheorem{example}[thm]{Example}

\renewcommand{\tilde}{\widetilde}

\newcommand{\sF}{{\mathcal F}}
\newcommand{\sG}{{\mathcal G}}
\newcommand{\sH}{{\mathcal H}}

\newcommand{\sL}{{\mathcal L}}

\newcommand{\sS}{{\mathcal S}}

\newcommand{\sX}{{\mathcal X}}

\newcommand{\A}{{\mathbb A}}

\newcommand{\C}{{\mathbb C}}

\newcommand{\G}{{\mathbb G}}

\renewcommand{\P}{{\mathbb P}}

\newcommand{\Z}{{\mathbb Z}}

\input{xy}
\xyoption{all}

\title[Rationality of varieties and $\mathbb{A}^1$-homotopy theory]{Rationality of varieties and $\mathbb{A}^1$-homotopy theory}
\author{Utsav Choudhury}
\address{Indian Statistical Institute\\ Stat Math Unit\\203 B.T. Road \\ Kolkata - 700108\\ India}
\email{utsav@isical.ac.in}
\author{Biman Roy}
\address{University of Southern California\\Department of Mathematics\\3620 Vermont Avenue\\Los Angeles, California 90089-2532}
\email{bimanroy31@gmail.com, bimanroy@usc.edu}
\begin{document}
\maketitle
\begin{abstract}
In this article we give an account of the results relating different rationality properties with $\mathbb{A}^1$-homotopy theoretic invariants.
\end{abstract}
\tableofcontents
\section{Introduction}
Let $k$ be a field and $Sm/k$ be the category of smooth finite type $k$-schemes. The unstable motivic homotopy category constructed by Morel-Voevodsky \cite{mv}, denoted by $\sH(k)$, is obtained by attaching simplicial homotopy to $Sm/k$ to obtain the category of simplicial presheaves, denoted by  $\bigtriangleup^{op}PSh(Sm/k)$,  and then inverting the Nisnevich local equivalences and the projection maps $pr : X \times \mathbb{A}^1 \to X$ for any $X \in Sm/k$. Objects in $\sH(k)$ are the simplicial presheaves on $Sm/k$ and any such object $\mathcal{X}$ will be called a space. For example, any $X \in Sm/k$ gives an object in $\sH(k)$. Recall that a morphism of spaces $f : \mathcal{X} \to \mathcal{Y}$ is called an $\mathbb{A}^1$-weak equivalence if the class of $f$ in $\sH(k)$ is an isomorphism.There is a pointed version of this construction, denoted by $\sH_{\bullet}(k)$, is called the pointed unstable $\mathbb{A}^1$-homotopy category. This homotopy category is constructed from the category of pointed spaces after inverting $\mathbb{A}^1$-weak equivalences of pointed spaces. A morphism of pointed spaces $f : (\mathcal{X}, x) \to (\mathcal{Y}, y)$ is an $\mathbb{A}^1$-weak equivalence if the corresponding morphism $f : \mathcal{X} \to \mathcal{Y}$, after forgetting the base point, is an isomorphism in $\sH(k)$. 

\

 For any space $\mathcal{X}$, define $\pi_0^{\mathbb{A}^1}(\mathcal{X})$ to be the Nisnevich sheaf associated to the presheaf 
$$\pi_{0, pre}^{\mathbb{A}^1}(\mathcal{X})(U) := Hom_{\sH(k)}(U, \mathcal{X}), \ \text{for }U \in Sm/k.$$
This is called the $\mathbb{A}^1$-connected component sheaf of the space $\mathcal{X}$. In this article we will survey the results relating $\mathbb{A}^1$-connected component sheaf of a variety $X$ and  rationality properties of $X$. We will give a brief account of moduli of $\mathbb{A}^1$-contractible varieties in every dimension.




\

 We will also give a brief survey of moduli of varieties $\mathbb{A}^1$-homotopic to mixed spheres. There are two circles $S^1_s$ and $S^1_t$, the simplicial and the Tate circles respectively, in $\sH(k)$ (\cite[Section 3.2]{mv}). These give us the following spheres for $i \geq 0$, $$S^i_s := (S^1_s)^{\wedge i}, \ S^i_t := (S^1_t)^{\wedge i}$$ and for $i \geq j \geq 0$ we have the following mixed spheres $$S^{i,j} := S_s^{i-j} \wedge S^j_t.$$  The quasi-affine varieties $\mathbb{A}^n_k \setminus \left\{(0,\dots, 0) \right\}$ are also mixed spheres in $\sH_{\bullet}(k)$, infact $\mathbb{A}^n_k \setminus \left\{(0,\dots, 0) \right\} \cong S^{2n-1, n}$ in $\sH_{\bullet}(k)$ (\cite[Section 3.2, Example 2.20]{mv}).

\


\noindent\textbf{Acknowledgements:} Choudhury was supported by  ANRF/ARGM project
ANRF/ARGM/2025/000240/MTR. A part of this material is based on work supported by the National Science Foundation under Grant No. DMS-2424139, while the author was in residence at the Simons Laufer Mathematical Scineces Institute in Berkeley, California, during the Fall 2026 semester. 
The second author thanks Aravind Asok for his continuous encouragement and the discussions. 
The second author would like to thank the University of Southern California for providing the resources during this work.

\section{$\mathbb{A}^1$-rigidity of $\mathbb{A}^1$-connected component}

In this section we discuss several general $\mathbb{A}^1$-rigidity results of the $\mathbb{A}^1$-connected component sheaf $\pi_0^{\mathbb{A}^1}(\mathcal{X})$ of a space $\mathcal{X}$.

\begin{definition} \cite[Definition 1.7]{mor}
Suppose, $\sF$ is a presheaf of sets on $Sm/k$.
\begin{enumerate}
\item $\sF$ is called $\A^1$-rigid, if for any $U \in Sm/k$, the morphism 
$$\sF(U) \to \sF(\A^1_U) \text{ induced by the projection } \A^1_U \to U,$$
is a bijection.
\item $\sF$ is called $\A^1$-rigid in one variable, if for every finitely generated, separable field extension $F/k$, the morphism 
$$\sF(Spec \ F) \to \sF(\A^1_F), \text{ induced by the projection } \A^1_F \to Spec \ F,$$
is a bijection.
\end{enumerate}

\end{definition}

From the construction of the $\A^1$-homotopy category $\sH(k)$, the presheaf $\pi_{0, pre}^{\mathbb{A}^1}(\mathcal{X})$ is $\A^1$-rigid for any space $\mathcal{X}$. Inspired by the discrete nature of the topological connected components associated to a topological space, Morel asked whether the Nisnevich sheafification $\pi_0^{\mathbb{A}^1}(\mathcal{X})$ is also an $\mathbb{A}^1$-rigid sheaf, for any space $\sX$ \cite[Conjecture 1.12]{mor}. Ayoub constructed a space $\sX$ over $k$, for which $\pi_0^{\A^1}(\sX)$ is not an $\A^1$-rigid sheaf \cite{ay}. However, such $\sX$ is not a representable sheaf (see also Theorem \ref{collection of A1-invariance results}(6)). Therefore, Morel's conjecture is false in general. There are some instances in which $\pi_0^{\A^1}(\sX)$ satisfies the $\A^1$-rigidity properties.


\begin{theorem} \label{collection of A1-invariance results}
\begin{enumerate}

\item For any space $\mathcal{X}$, $\pi_0^{\mathbb{A}^1}(\mathcal{X})$ is $\A^1$-rigid in one variable.
\item If $\pi_0^{\mathbb{A}^1}(\mathcal{X})$ is a separated $k$-scheme (or it is just a representable sheaf), then the $\A^1$-rigidity in one variable is equivalent to $\mathbb{A}^1$-rigidity. Thus in this case, $\pi_0^{\A^1}(\sX)$ is $\A^1$-rigid.
\item For homogeneous $H$-spaces $\mathcal{X}$, $\pi_0^{\mathbb{A}^1}(\mathcal{X})$ is $\mathbb{A}^1$-rigid.
\item If $\sX$ is a smooth toric variety over $k$, then $\pi_0^{\A^1}(\sX)$ is $\A^1$-rigid.
\item If $X$ is a surface, which is either affine or projective over an algebraically closed field of characteristic zero then $\pi_0^{\A^1}(\sX)$ is $\A^1$-rigid.
\item Let $X$ be a smooth proper $\mathbb{A}^1$-rigid threefold and let $Y \to X$ be a closed surface such that Y has an isolated singularity at a $k$-point $y \in Y(k) $ admitting a resolution $b : Y'
\to Y$ such that $C = b^{-1}(y)$ is a cuspidal rational curve. Then there exists a space $\mathcal{M}_{X,Y}$ such that $\pi_0^{\mathbb{A}^1}(\mathcal{M}_{X,Y})$ is not $\mathbb{A}^1$-rigid.
\end{enumerate}
\end{theorem}

\begin{proof}
\begin{enumerate}
\item The result is proved in \cite[Corollary 3.2]{u}. The main goal is to show that for any essentially smooth curve $C$ over $k$, the map 
$$\pi_{0, pre}^{\mathbb{A}^1}(\mathcal{X}) (C) \to \pi_0^{\mathbb{A}^1}(\mathcal{X})(C)$$ 
is surjective. The surjection intuitively follows from the fact that the Nisnevich cohomological dimension is bounded by the Krull dimension. In the proof of \cite[Corollary 3.2]{u}, the main ingredient  is that Nisnevich Brown-Gersten property induces a surjection at $\pi_0$ \cite[Lemma 2.2]{u}: if $\sX$ is a Nisnevich fibrant space (in particular, $\A^1$-fibrant) and if the following pullback square 
  
        \[
        \xymatrix{
U \times_X V \ar[d] \ar[r] &V \ar[d]^p \\
        	U \ar[r]^j & X }
        \]
 is an elementary distinguished square \cite[Definition 1.3, Section 3.1]{mv}, then the induced morphism
$$\pi_0(\sX(X)) \to \pi_0(\sX(U)) \times_{\pi_0(\sX(U \times_X V))} \pi_0(\sX(V))$$       
is a surjection (\cite[Definition 1.13 and Remark 1.15, Section 3.1]{mv}, \cite[Lemma 2.2]{u}). When $C$ is spectrum of a DVR, the surjection follows from this Brown Gersten property of $\pi_0$. For a curve $C$ the surjection then follows from Noetherian induction. 

\item Suppose $\pi_0^{\A^1}(\sX)$ is given by a separated scheme $Y$. If $Y$ is $\A^1$-rigid, then it also $\mathbb{A}^1$-rigid in one variable. Note that, a sheaf $\mathcal{F}$ on $Sm/k$ is $\mathbb{A}^1$ rigid if for any smooth scheme $U$ and any  section $s \in \mathcal{F}(\mathbb{A}^1_U)$, the maps induced by the $0$-section and $1$-section are equal, that is $s(0) = s(1) \in \sF(U)$ \cite[Lemma 2.16]{MWV}. Indeed, suppose that for any $U \in Sm/k$ for any $s \in \sF(\A^1_U)$, we have $s(0) = s(1)$. Consider the morphism

$$\Phi: \A^1_U \times_k \A^1_k \to \sF \text{ given by }(\alpha,x,\beta) \mapsto s(\alpha\beta, x).$$
The section $\Phi$ satisfies
$$\Phi(\alpha,x,0) = \Phi(\alpha, x, 1) \in \sF(\A^1_U) \text{ (by hypothesis) and }$$
$$\Phi(\alpha, x, 0) = s \circ i_0 \circ \text{pr}, \ \Phi(\alpha, x, 1) = s.$$
Therefore given any such $s$ factors as 
$$\mathbb{A}^1_U \xrightarrow{\text{pr}} U \xrightarrow{s(0)} \mathcal{F}.$$

Suppose the separated scheme $Y$ is also $\A^1$-rigid in one variable. Let $s : \mathbb{A}^1_U \to Y$ be a morphism. We need to prove that 
$$s(0) = s(1) : U \to Y.$$
Note that for any $u \in U$, the composition of the morphisms
$$\A^1_{\kappa(u)} \xrightarrow{(Id,u)} \A^1_U \xrightarrow{s} Y$$
factors through $Spec \ \kappa(u)$. Thus
we have, 
$$s(0,u) = s(1,u) : Spec(\kappa(y)) \to Y.$$
Let $\delta : Y \to Y \times_k Y$ be the diagonal embedding and 
$$S := (s(0), s(1)) : U \to Y \times_k Y$$ 
be the morphism obtained from $s(0)$ and $s(1)$. Then the pullback $S^*(\delta)$ is a closed subscheme of $U$ that contains the closed points $u$ of $U$ such that the residue field $k(u)/k$ is a separable extension. Since such points are dense in $U$ \cite[Tag 056U]{stack}, therefore $S^*(\delta)$ equals to $U$ and we are done. Here for any scheme $Y$, since the diagonal $Y \to Y \times_k Y$ is locally closed and we would get the same conclusion. Since $\pi_0^{\A^1}(\sX)$ is $\A^1$-rigid in one variable, therefore if $\pi_0^{\A^1}(\sX)$ is representable, then it is $\A^1$-rigid. For example, if $\sX$ is $\A^1$-connected, $\sX$ is a proper, open subscheme of $\A^1_k$, $\sX$ is an abelian variety or $\sX$ is a smooth projective curve of positive genus, then $\pi_0^{\A^1}(\sX)$ is representable, hence it is $\A^1$-rigid.

\item See \cite[Theorem 4.18]{u} for the proof. The main ingredient here is that in case of homogeneous $H$-spaces $\sX$ on $Sm/k$ there is a distinguished base point over $k$ which can be moved (homotoped) to any point over an essentially smooth  Hensel local $k$-scheme $U$. Then we apply Gabber  presentation to conclude that for any smooth scheme $U$ and any open subcheme $V \subset U$, the restriction map of the $\mathbb{A}^1$-conneceted component sheaf is injective, therefore we are reduced to the $\mathbb{A}^1$-rigidity in single variable.
\item This is proved in \cite[Lemma 4.2, Lemma 4.4]{w}. It is shown if $\sX$ is a toric variety over $k$, then $\pi_0^{\A^1}(\sX)$ is isomorphic to $\sS(X)$ (Definition \ref{geometric sheaf}). In this case, either $\sX$ is $\A^1$-connected (in this case, $\sS(X)$ is the trivial sheaf) or $\pi_0^{\A^1}(\sX)$ is isomorphic to the product of $\G_m$'s.
\item The projective case is proved in \cite[Corollary 3.15]{bhs} (for non-uniruled case), \cite[Theorem 1.2]{bs} (for the birationally ruled case) and the affine case is proved in \cite[Corollary 4.40]{aritra} (see Section \ref{rationality component} for the details).
\item This is \cite[Construction 8]{ay}.
\end{enumerate}
\end{proof}

\begin{remark} \label{remarks on pi0A1}
\begin{enumerate}
\item The single variable $\mathbb{A}^1$-rigidity is a topological property of the  site with interval and on the other hand the $\mathbb{A}^1$-rigidity seems to be an internal ``geometric" property of the space. 
\item For a space $\sX$ over $k$, the sheaf $\pi_0^{\A^1}(\sX)$ satisfies the following universal property \cite[Lemma 2.8]{bhs}: any morphism $\sX \to \sF$ to an $\A^1$-rigid sheaf $\sF$ uniquely factors through the canonical epimorphism $\sX \to \pi_0^{\A^1}(\sX)$.
\item An unramified sheaf of sets \cite[Definition 2.1]{mor} is $\A^1$-rigid, if and only if it is $\A^1$-rigid in one variable \cite[Lemma 2.16]{mor}. However, $\pi_0^{\A^1}(\sX)$ is not an unramified sheaf of sets in general.


Suppose, $\Tilde{C}$ is the smooth, non-separated curve obtained by gluing two copies of $C:= \A^1_k \setminus \{1\}$ along $\A^1_k \setminus \{0,1\}$. Then $\pi_0^{\A^1}(\Tilde{C})$ is not unramified sheaf of sets. 
Indeed by the construction, $\Tilde{C}$ is the pushout of the diagram (in the category of Nisnevich sheaves on $Sm/k$)
\[
\xymatrix{
\A^1_k \setminus \{0,1\} \ar[r] \ar[d] & C \ar[d]^{i_1} \\
C \ar[r]^{i_2} & \Tilde{C}
}
\]
so, the morphism $\Tilde{C}(C) \to \Tilde{C}(\A^1_k \setminus \{0,1\})$ is not injective. The non-separated curve $\Tilde{C}$ is $\A^1$-rigid. Indeed, if there is a non-constant morphism $\A^1_k \to \Tilde{C}$, then it must be surjective, as $C$ is $\A^1$-rigid. But the surjective map would induce an injective $k$-algebra morphism between the rings of regular functions $k[t]_{t-1} \to k[t]$. The unit element $t-1 \in k[t]_{t-1}$ would map to a non-zero element in $k$, this would contradict the injectivity. The $\A^1$-rigidity of $\Tilde{C}$ implies that $\pi_0^{\A^1}(\Tilde{C}) \cong \Tilde{C}$ 
is not an unramified sheaf of sets.
Consider the example
$$X \equiv \{x^2z=y(y-1)+x\} \subseteq C \times_k \A^2_k$$ 
is an smooth, affine surface over $k$, which admits an $\A^1$-fibration \cite[Section 1]{doub} 
$$\text{pr}: X \to C \text{ given by }(x,y,z) \mapsto x.$$
Every fiber of $\text{pr}$ is $\A^1$, except the fiber over $0$ is isomorphic to the disjoint union of two copies of $\A^1$'s. So the induced map $\pi: X \to \Tilde{C}$, which maps two irreducible components of $\text{pr}^{-1}(0)$ to the two origins respectively, is an $\A^1$-bundle (in Zariski topology). So $\pi$ is an $\A^1$-weak equivalence and therefore, 
$$\pi_0^{\A^1}(X) \cong \Tilde{C} \text{ is not an unramified sheaf of sets}.$$  

It is natural to ask when $\pi_0^{\A^1}(X)$ is an unramified sheaf of sets. 
If $\pi_0^{\A^1}(X)$ is unramified, then it is $\A^1$-rigid.
\item For a homogeneous $H$-space $\sX$, $\pi_0^{\A^1}(\sX)$ is $\A^1$-rigid. The restriction map $\pi_0^{\A^1}(\sX)(U) \to \pi_0^{\A^1}(\sX)(V)$ is injective, for an open subscheme $V$ of $U \in Sm/k$ \cite[Corollary 4.18]{u}. One can ask whether $\pi_0^{\A^1}(\sX)$ is unramified sheaf of sets \cite[Definition 2.1]{mor} or equivalently, the restriction $\pi_0^{\A^1}(\sX)(U) \to \pi_0^{\A^1}(\sX)(V)$ is bijective, for every $V \subseteq U$ an open subscheme such that $U \setminus V$ is of codimension $\geq 2$ in $U$ \cite[Remark 2.4]{mor}. If $\sX$ is a semisimple, simply connected algebraic group over perfect field (for examples $SL_n(k), \ n\geq 2$, $Sp_{2n}(k)$, $n \geq 1$), then $\pi_0^{\A^1}(\sX)$ is a birational sheaf of groups (in particular unramified, \cite[Theorem 4.5]{bhs2}, see also \cite{bhs2corri}).
\end{enumerate}
\end{remark}

\section{Rationality and $\mathbb{A}^1$-connected components} \label{rationality component}
Suppose, $X \in Sm/k$. In this section, we discuss the comparisons between the $\A^1$-connected component sheaf $\pi_0^{\A^1}(X)$ and the $\A^1$-chain connected component sheaf $\pi_0^{ch}(X)$ (Definition \ref{chain connected}). As consequences, we discuss the rationality properties of $X$, related to the $\A^1$-connected component sheaf $\pi_0^{\A^1}(X)$ of $X$. Throughout the section we assume $k$ to be an algebraically closed field of characteristic $0$.
\begin{definition} \cite[Section 2.3]{mv}
For a presheaf of sets $\sF$ on $Sm/k$, the simplicial presheaf $Sing_*(\sF)$ is defined as
$$Sing_*(\sF)(U)=\sF(U \times_k \Delta^n_a), \text{ for }U \in Sm/k,$$
where $\Delta^n_a:= Spec \ \frac{k[x_0,x_1, \dots,x_n]}{\sum_{i=0}^n x_i-1}$, $n \geq 0$ is the cosimplicial object in $Sm/k$ and it is isomorphic to $\A^n_k$.
\end{definition}
\begin{definition}\cite[Definition 2.2.4]{am} \label{chain connected}
The $\mathbb{A}^1$-chain connected component sheaf of $X$, denoted by $\pi_0^{ch}(X)$, is defined to be the connected component sheaf of $Sing_*(X)$, that is it is the Nisnevich sheaf associated to the presheaf of sets
$$U \mapsto \pi_0(Sing_*(X)(U)).$$
More precisely, the section $\pi_0(Sing_*(X)(U))$ is actually the quotient $X(U)/\sim$ associated to the equivalence relation ``$\sim$" on $X(U)$ given by for two morphisms $f,g:U \to X$, $f \sim g$ if there are $H_1, \dots, H_n: \A^1_U \to X$ such that
$$H_i(1)=H_{i+1}(0): U \to X, \ H_1(0)=f \text{ and }H_n(1)=g,$$
where $H_i(0), H_i(1)$ are induced by the $0$-section, $1$-section $U \to \A^1_U$ respectively (see also Definition \ref{geometric sheaf}).

$X$ is called (weakly) $\A^1$-chain connected, if the set $\pi_0^{ch}(X)(Spec \ F)$ is a singleton set, for every finitely generated separable field extension $F/k$.
\end{definition}    
\begin{remark}
There is a canonical morphism $\pi_0^{ch}(X) \to \pi_0^{\A^1}(X)$, which is an epimorphism \cite[Lemma 2.2.5]{am}. Therefore by a result of Morel, if $X$ is weakly $\A^1$-chain connected, then $X$ is $\A^1$-connected \cite[Corollary 2.2.6]{am}.
\end{remark}

\begin{theorem}\label{rationally connected}
For a proper variety $X \in Sm/k$, the sections of $\pi_0^{ch}(X)$ and $\pi_0^{\A^1}(X)$ over $Spec \ F$ are same, for every finitely generated separable field extension $F/k$. Thus $X$ is $\A^1$-connected, if and only if $X$ is weakly $\A^1$-chain connected.
\end{theorem}

\begin{proof}
This is proved in \cite[Theorem 2.4.3]{am} (and also \cite[Corollary 3.10]{bhs}). We discuss a sketch here.

In \cite[Section 6]{am}, Asok-Morel constructed a birational and $\A^1$-rigid sheaf ($\A^1$-rigidity is redundant, see Remark \ref{rational connectedness remark})
$\pi_0^{b\A^1}(X)$, associated to a proper scheme $X \in Sm/k$ such that its section over $Spec \ F$
agrees with the section $\pi_0^{ch}(X)(Spec \ F)$, for every finitely generated separable field extension $F/k$. For $X$ proper, the natural map $X(O_v) \rightarrow X(\kappa(v))$, together  with the bijection $X(O_v) \cong X(K(O_v))$ induces by \cite[Lemma 6.2.3]{am} a residue morphism $$\pi_0^{b\A^1}(X)(Spec \ F) \to \pi_0^{b\A^1}(X)(Spec \ \kappa_v),$$ where $K(O_v) = F$ is a finitely generated separable field extension of $k$, $v$ is a discrete valuation on $F$ with $O_v$ its associated DVR and $\kappa(v)$ the residue field.
A birational sheaf is equivalent with these datas \cite[Proposition 6.2.4]{am}.
 
The universal property of $\pi_0^{\A^1}(X)$ (see Remark \ref{remarks on pi0A1}(2)) defines the morphism $\pi_0^{\A^1}(X) \to \pi_0^{b\A^1}(X)$ which gives bijection of the composition of the morphisms
$$\pi_0^{ch}(X) \to \pi_0^{\A^1}(X) \to \pi_0^{b\A^1}(X)$$
over the field valued sections. Since the morphism $\pi_0^{ch}(X) \to \pi_0^{\A^1}(X)$ is an epimorphism of sheaves, so the morphism $\pi_0^{ch}(X)(F) \to \pi_0^{\A^1}(X)(F)$ is a bijection.

The proof in \cite{bhs} is different. Balwe-Hogadi-Sawant constructed the universal $\A^1$-rigid sheaf $\sL(X)$ (Definition \ref{geometric sheaf}), associated to $X$ to be the filtered colimit of the $n$-th iterations of $\pi_0^{ch}(X)$ (following the notation in \cite{bhs}, $\sS(X):=\pi_0^{ch}(X)$). They proved that for proper $X \in Sm/k$, the iteration $\sS^n(X)$ shares following properties with the proper schemes: a morphism $U \to \sS^n(X)$, for $U \in Sm/k$ with $dim(U) \leq 2$, can be extended to some smooth, projective compactification $\bar{U}$ of $U$ on ``points" and for $C \in Sm/k$ of dimension $1$, $C^\prime \subseteq C$ an open subscheme, the restriction map $\sS^n(X)(C) \to \sS^n(X)(C^\prime)$ is injective on ``points" (see \cite[Definition 3.6, Lemma 3.8]{bhs}). They called the sheaf satisfying such properties to be an almost proper sheaf. They proved that
an existence of $\A^1$-ghost homotopy (Definition \ref{ghost homotopy definition}) in an almost proper sheaf $\sF$ between two field valued points would lift to a naive $\A^1$-homotopy in $\sF$ between the points \cite[Theorem 3.9]{bhs}. This implies that $\sS^n(X)(Spec \ F) \cong \sS(X)(Spec \ F)$, for every $n \geq 1$ and thus
the universal morphism
$$\pi_0^{\A^1}(X) \to \sL(X)$$
induces surjection $\pi_0^{\A^1}(X)(Spec \ F) \to \sS(X)(Spec \ F)$. Consequently,  the epimorphism 
$$\sS(X) \to \pi_0^{\A^1}(X)$$
induces bijection over the section $Spec \ F$, for every finitely generated separable field extension $F/k$.


\end{proof}
\begin{remark} \label{rational connectedness remark}
\begin{enumerate}
\item A birational presheaf of sets on $Sm/k$ is an $\A^1$-rigid sheaf (in Nisnevich topology). Firstly, it is an unramified sheaf of sets on $Sm/k$ (\cite[Lemma 6.1.2]{am}, \cite[Definition 2.1]{mor}). Also it is $\A^1$-rigid in one variable by \cite[Appendix A]{ks}. Therefore, it is $\A^1$-rigid \cite[Lemma 2.16]{mor}.
\item The sheaf $\pi_0^{b\A^1}(X)$ satisfies the following universal property: a morphism $X \to \sF$, where $\sF$ is a birational and $\A^1$-rigid sheaf, uniquely factors through the canonical morphism $X \to \pi_0^{b\A^1}(X)$ \cite[Remark 6.2.2]{am}. 
\item Suppose, 
$X \in Sm/k$ is a smooth, proper variety and $F/k$ is a finitely generated separable field extension. Theorem \ref{rationally connected} implies that if $X$ is $\A^1$-connected if and only if given any $x,y \in X(F)$, there are $\gamma_1, \dots, \gamma_n: \A^1_F \to X$ such that
$$\gamma_i(1)=\gamma_{i+1}(0), \text{ for every }i \text{ and }\gamma_1(0)=x, \ \gamma_n(1)=y,$$
where $\gamma_i(0), \gamma_i(1) \in X(F)$ are induced by the $0$-section and $1$-section $Spec \ F \to \A^1_F$ respectively. Thus the $\A^1$-connectedness of $X$ implies that $X$ is
also rationally connected \cite[Definition 3.2, Chapter IV]{kollar}.
Thus if moreover the base field is $k = \C$, then $X$ is uniruled and $X$ has negative Kodaira dimension (see also Theorem \ref{uniruledness generality}). So if $X$ is an $\A^1$-connected, smooth, proper surface over $\C$, 
then $X$ is rational. 

However, the rational connectedness is strictly a weaker notion than the $\A^1$-connectedness. There is a unirational (which admits a dominant map from a projective space, in particular it is rationally connected), non-rational, proper threefold over $\C$ \cite{artinmumford}, which is not $\A^1$-connected (as it has  non-trivial Brauer group). 
\end{enumerate}
\end{remark}
Here we first recall the definition of birational unstable motivic homotopy category from \cite[Definition 3.4]{cb} (see also \cite[Definition 2.6]{pp} and \cite{ck}) and in Theorem \ref{birational connected theorem}, we will see that the birational connected component $\pi_0^{b\A^1}(X)$, for a proper $X \in Sm/k$, is indeed the connected component sheaf in that homotopy category.
\begin{definition}
The left Bousfield localisation of the projective model structure on $\bigtriangleup^{op}PSh(Sm/k)$ with respect to the open dense immersions $U \hookrightarrow X$, $X \in Sm/k$ gives the unstable birational model structure. The associated homotopy category will be denoted by $\sH_b(k)$.

For a space $\sX$ over $k$, the connected component presheaf of sets (in the unstable birational model structure) is defined as
$$U \in Sm/k \mapsto Hom_{\sH_b(k)}(U, \sX),$$
and it is denoted by $\pi_0^b(\sX)$. Since it is a birational presheaf of sets on $Sm/k$ \cite[Definition 6.1.1]{am}, the presheaf $\pi_0^b(\sX)$ is actually a Nisnevich sheaf of sets \cite[Lemma 6.1.2]{am}.
\end{definition}
\begin{remark} \label{birational model remark}
\begin{enumerate}
\item An $\A^1$-weak equivalence is an isomorphism in $\sH_b(k)$. Indeed, the \v{C}ech hypercovers and the projection maps $\A^1_X \to X$, where $X \in Sm/k$ generate the $\A^1$-weak equivalences \cite[Proposition 8.1]{dugger}. A \v{C}ech hypercover associated to a Nisnevich covering is an isomorphism in $\sH_b(k)$
, since a Nisnevich covering has a section over a dense open subset. By definition, $\pi_0^b(\sX)$ is a birational presheaf of sets on $Sm/k$, so it is $\A^1$-rigid (Remark \ref{rational connectedness remark}).
So the projections $\A^1_X \to X$, $X \in Sm/k$ are also isomorphisms in $\sH_b(k)$.
\item Any birational sheaf is a fibrant object in the birational model structure. So the connected component $\pi_0^b(\sX)$ satisfies the same universal property as $\pi_0^{b\A^1}(X)$ (see Remark \ref{rational connectedness remark}(2)): given  a birational sheaf of sets $\sF$ on $Sm/k$, the morphism of sets
$$Hom_{PSh(Sm/k)}(\pi_0^b(\sX), \sF) \to Hom_{PSh(Sm/k)}(\sX, \sF)$$
induced by the canonical map $\sX \to \pi_0^b(\sX)$, is a bijection (here $PSh(Sm/k)$ is the category of presheaves of sets on $Sm/k$).
\end{enumerate}
\end{remark}
\begin{theorem} \label{birational connected theorem} (\cite[Theorem 3.5]{cb}, see also \cite[Proposition 1.9]{ck})
The birational connected component of a proper variety is its connected component in birational homotopy theory.

\end{theorem}
\begin{proof}
The proof follows because both the sheaves $\pi_0^{b\A^1}(X)$ and $\pi_0^b(X)$ satisfy the same universal property (see Remark \ref{rational connectedness remark} and Remark \ref{birational model remark}).
\end{proof}
A variety $X \in Sm/k$ of dimension $n$ is called rational, if there is a birational map $\P^n_k \dashrightarrow X$. A rational variety has rational curves between any two general points. If it is also proper, then that rational curve can be completed to a morphism from $\P^1$. The following result says that the proper, rational varieties are $\A^1$-chain connected.

\begin{theorem} \label{rational implies connected}
Suppose, $k$ is a field of characteristic zero.  A proper, rational variety $X \in Sm/k$ is $\mathbb{A}^1$-chain connected, so it is $\A^1$-connected. Therefore, if moreover $X$ is a surface and $k$ is algebraically closed, then $X$ is rational if and only if it is $\A^1$-connected.

\end{theorem}

\begin{proof}
This is proved in \cite[Theorem 2.3.6]{am} (also \cite[Corollary 2.4.7]{am}). 
The weak factorization holds in characteristic zero, that is a rational, proper variety is obtained by a sequence of blow-ups and blow-downs at a smooth center from $\mathbb{P}^n$.
Finally the proof follows because if $\pi: Y^\prime \to Y$ is the blow-up at a smooth center,
then $Y$ is $\A^1$-chain connected if and only if $Y^\prime$ is $\A^1$-chain connected. 

The other assertion follows because smooth $\A^1$-connected proper varieties are rationally connected and in dimension $2$ over an algebraically closed field of characteristic $0$, the rational connectedness is equivalent to being rational.

\end{proof}
\begin{remark}
\begin{enumerate}
\item Like the rational connectedness, the property of being $\A^1$-connected is a birational invariant for smooth proper varieties over the field of characteristic $0$ \cite[Corollary 2.4.6]{am}.
\item Theorem \ref{rational implies connected} is more generally true. A retract rational (a variety is called retract rational if it contains an open subset, which is a retract of some open subset of the affine space), proper variety $X$ over any field is weakly $\A^1$-chain connected (see \cite[Theorem 2.3.6]{am} and \cite[Theorem 8.5.1, Theorem 8.6.2]{ks}). Moreover if $k$ is an infinite field, then the $\A^1$-chain connected component sheaf $\pi_0^{ch}(X)$ is the trivial sheaf \cite[Theorem 1.1]{retract}.
\end{enumerate}
\end{remark}
\begin{remark}
As we have discussed earlier (Remark \ref{rational connectedness remark}), in higher dimensions it is not known whether an $\A^1$-connected, smooth, proper variety is rational. If $X \in Sm/k$ is $\A^1$-connected, then the unramified cohomology groups $H^i_{ur}(X, \mu_n^{\otimes l})$ vanish, for $i \geq 1$ \cite[Lemma 4.7]{asok homology}. Thus in particular, an $\A^1$-connected smooth, proper variety has trivial Brauer group.
\end{remark}
Here we recall several classes of varieties $X$ with lots of rational curves or affine lines in $X$. By an affine line in $X$ (or a rational curve), we mean that there is a non-constant morphism $\A^1_k \to X$ (or there is a non-constant rational map $\A^1_k \dashrightarrow X$).
\begin{definition} (\cite[Definition 1.1, Chapter IV]{kollar}, \cite[Definition 1]{at})
Suppose, $X \in Sm/k$ is a variety over an algebraically closed field of characteristic zero.
\begin{enumerate}
\item The variety $X$ is called $\A^1$-uniruled or log-uniruled (or simply uniruled, if $X$ is proper) if there is a dominant generically finite morphism $\A^1_Y \to X$, for some $k$-variety $Y$.
\item The variety $X$ is called $\A^1$-ruled (or $\A^1$-cylindrical) if there is an open subset $U$ of $X$ which is isomorphic to $\A^1_k \times_k Y$, for some $k$-variety $Y$.
\item The variety $X$ is called ruled if it is birational to $\A^1_k \times_k Y$, for some $k$-variety $Y$.
\end{enumerate}
\end{definition}
\begin{remark}
\begin{enumerate}
\item An $\A^1$-ruled variety is also $\A^1$-uniruled and an $\A^1$-uniruled variety has negative logarithmic Kodaira dimension.
\item An $\A^1$-ruled variety is also ruled. However, in general whether a smooth rational projective variety contains an $\A^1$-cylinder is not known (see \cite{at}, \cite{kishimoto}).
\end{enumerate}
\end{remark}

Now we recall the universal $\A^1$-invariant sheaf $\sL(\sF)$, associated to a sheaf $\sF$, introduced by Balwe-Hogadi-Sawant in \cite{bhs}, which is geometric in nature and is related to the $\A^1$-connected component $\pi_0^{\A^1}(\sF)$.
\begin{definition} \cite[Definition 2.9]{bhs} \label{geometric sheaf}
Suppose, $\sF$ is a Nisnevich sheaf of sets on $Sm/k$. The sheaf $\sS(\sF)$ is defined to be the Nisnevich sheafification of the presheaf of sets $\sS^{\text{pre}}(\sF)$ on $Sm/k$, defined as
$$\sS^{\text{pre}}(\sF)(U)= \sF(U)/\sim, \text{ for }U \in Sm/k,$$
where $\sF(U)/\sim$ is the quotient associated to the equivalence relation $\sim$ on $\sF(U)$, generated by the naive $\A^1$-homotopies (see Definition \ref{chain connected}).

If $n >1$ is an integer, then the sheaf $\sS^n(\sF)$ is defined inductively as $\sS^n(\sF) := \sS(\sS^{n-1}(\sF))$ and $\sS^0(\sF):= \sF$. The quotient map associated to the equivalence relation $\sim$ induces an epimorphism of Nisnevich sheaves
$$\theta: \sF \to \sS(\sF).$$
The sheaf $\sL(\sF)$ is defined to the filtered colimit
$$\sL(\sF):= \varinjlim_{n \geq 0} \sS^n(\sF).$$
There is an canonical epimorphism $\sF \to \sL(\sF).$
\end{definition}
\begin{remark}
\begin{enumerate}
\item If $X \in Sm/k$, then $\sS(X)$ is the $\A^1$-chain connected component sheaf $\pi_0^{ch}(X)$ (Definition \ref{chain connected}, \cite[Remark 2.11]{bhs}).
\item For a Nisnevich sheaf $\sF$, the canonical epimorphism $\sF \to \pi_0^{\A^1}(\sF)$ factors through the canonical epimorphism $\sF \to \sS(\sF)$.
\item The sheaf $\sL(\sF)$ is $\A^1$-rigid \cite[Theorem 2.13]{bhs}. The sheaf $\sL(\sF)$ satisfies the same universal property as $\pi_0^{\A^1}(\sF)$ \cite[Remark 2.15]{bhs}: a morphism $\sF \to \sG$ to an $\A^1$-rigid sheaf $\sG$ uniquely factors through the canonical morphism $\sF \to \sL(\sF)$. This induces the canonical epimorphism $\pi_0^{\A^1}(\sF) \to \sL(\sF)$, which is an isomorphism if and only if $\pi_0^{\A^1}(\sF)$ is $\A^1$-rigid \cite[Corollary 2.18]{bhs}.
\end{enumerate}
\end{remark}
\begin{proposition}
For any space $\mathcal{X}$, the morphism $\pi_0^{\mathbb{A}^1}(\mathcal{X}) \to \mathcal{L}( \pi_0^{\mathbb{A}^1}(\mathcal{X}))$ is a bijection over the section $Spec \ F$, for any finitely generated separable field extension $F/k$.
\end{proposition}
\begin{proof}
See \cite[Remark 2.15]{bhs}, \cite[Corollary 2.2]{brs} (see also \cite[Corollary 2.15]{cb}).
\end{proof}
\begin{definition} \label{ghost homotopy definition}
Suppose, $X \in Sm/k$.
\begin{enumerate}
\item An element $H \in \sS(X)(\A^1_U)$, for some $U \in Sm/k$ (called an $\A^1$-ghost homotopy \cite[Definition 3.2]{bhs}) is given by a Nisnevich covering $f: V \to \A^1_U$ along with a morphism $\phi: V \to X$ such that there exists a Nisnevich covering $V^\prime \to V \times_{\A^1_U} V$ along with a chain of naive $\A^1$-homotopies $H_1, \dots, H_m: \A^1_{V^\prime} \to X$ such that 
$$H_1(0) = \phi \circ  \text{pr}_1|_{V^\prime}  \text{ and }H_m(1)= \phi \circ  \text{pr}_2|_{V^\prime},$$ 
where $\text{pr}_1, \text{pr}_2: V \times_{\A^1_U} V \to V$ are the projection maps and the elements $H_i(0), H_i(1): V^\prime \to X$ are induced by the $0$-section and the $1$-section $V^\prime \to \A^1_{V^\prime}$ respectively, . If moreover, the element 
$$H \in \sS(X)(\A^1_U) \text{ satisfies that } H(0) \neq H(1) \in \sS(X)(U),$$
then we can assume that each $H_i: \A^1_{V^\prime} \to X$ also satisfies that $H_i(0) \neq H_i(1): V^\prime \to X$, for every $i$ (see also \cite[Remark 2.6]{choudhury2}).

Two elements $\alpha, \beta: U \to X$ are $\A^1$-ghost homotopic if there is such $H \in \sS(X)(\A^1_U)$ along with the lifts $\Tilde{i_0}, \Tilde{i_1}: U \to V$ of the $0$-section, $1$-section $i_0, i_1: U \to \A^1_U$ respectively (that is, $f \circ \Tilde{i_0} = i_0$ and $f \circ \Tilde{i_1}=i_1$) such $\theta \circ \phi \circ \Tilde{i_0} = \theta \circ \alpha$ and $\theta \circ \phi \circ \Tilde{i_1} = \theta \circ \beta$, where $\theta: X \to \sS(X)$ is the canonical epimorphism. 
\item A finite collection of homotopies $H_1, \dots, H_m: \A^1_W \to X$, for some $W \in Sm/k$ is called a chain if they satisfy 
$$H_i(1)=H_{i+1}(0): W \to X, \text{ for every }i.$$
\end{enumerate}

\end{definition}

\begin{theorem} \label{proper surface rigid}
Suppose, $X \in Sm/k$ is a smooth proper surface, over an algebraically closed field of characteristic $0$.
Then $\pi_0^{\mathbb{A}^1} (X)$ is $\mathbb{A}^1$-rigid.
\end{theorem}

\begin{proof}

If $X \in Sm/k$ is non-uniruled proper surface over any field $k$, then the map $\sS(X) \to \sS^2(X)$ is an isomorphism \cite[Theorem 3.14]{bhs}. Thus, in this case $\sS^n(X)$ stabilizes at $n=1$ and $\sL(X) \cong \sS(X)$. So we have
$$\pi_0^{\A^1}(X) \cong \sS(X) \cong \sL(X), \text{ is }\A^1\text{-rigid}.$$
Let $k$ be an algebraically closed field of characteristic $0$. If $X$ is a projective surface over $k$, 
then the morphism
$$\sS(X)(U) \to \pi_0^{\A^1}(X)(U),$$
is an isomorphism, for Henselian local $k$-scheme $U$ of dimension $1$ \cite[Theorem 1]{naiveruled}. 
If $X$ is rational, proper surface over any infinite field $k$, then $\sS(X) \cong \pi_0^{\A^1}(X) \cong Spec \ k$.
If $X \in Sm/k$, 
is a proper, birationally ruled surface over a curve $C$ of genus $g>0$, then by \cite[Theorem 1.5]{bs}
\begin{enumerate}
\item If $X$ is its minimal model, then $\pi_0^{\A^1}(X) \cong \sS^n(X) \cong C$, for every $n \geq 1$. So $\pi_0^{\A^1}(X) \cong \sL(X) \cong \sS(X)$ is $\A^1$-rigid.
\item If $X$ is not a minimal model, then $\pi_0^{\A^1}(X) \cong \sS^n(X)$, for $n \geq 2$ and $\sS(X)\ncong \sS^2(X)$. So $\pi_0^{\A^1}(X) \cong \sL(X) \cong \sS^2(X)$ is $\A^1$-rigid.
\end{enumerate}
\end{proof}
\begin{remark} \label{non local}
For $X \in Sm/k$, the $\A^1$-locality of $Sing_*(X)$ \cite[Definition A.2.1]{am} is an useful condition to understand the $\A^1$-fibrant replacement functor geometrically. If $Sing_*(X)$ is $\A^1$-local, then the natural map $\sS(X) \to \pi_0^{\A^1}(X)$ is an isomorphism \cite[Remark 2.2.9]{am}.
\begin{enumerate}
\item If $X \in Sm/k$, where $k$ is an algebraically closed field of characteristic $0$, is proper, birationally ruled surface over a curve $C$ of genus $g>0$, which is not a mininal model, then the map $\sS(X) \to \pi_0^{\A^1}(X)$ is not an ismorphism, in particular $Sing_*(X)$ is not $\A^1$-local \cite[Theorem 2]{naiveruled}.
\item If $X \in Sm/k$ is an affine surface over an algebraically closed field of characteristic $0$, which admits an $\A^1$-fibration over a smooth $\A^1$-rigid curve, then $Sing_*(X)$ is $\A^1$-local \cite[Corollary 4.41]{aritra}.
\end{enumerate}
\end{remark}
Theorem \ref{rationally connected} implies that if $X$ is a smooth, proper variety which is also $\A^1$-connected, then $X$ is rationally connected. In particular, $X$ is $\A^1$-uniruled. The same assertion is true, for any $\A^1$-connected smooth variety.
\begin{theorem} \label{uniruledness generality}
Let $X$ be a smooth $\mathbb{A}^1$-connected $k$-variety. Then there exists a dominant morphism $H: \mathbb{A}^1_Y \to X$, for some $Y \in Sm/k$ with $dim(Y) = dim(X) -1$. 
\end{theorem}
\begin{proof}
This is proved in \cite[Theorem 2.11]{choudhury2}.

\end{proof}
\begin{corollary} \cite[Proposition 11.4]{iitaka} \label{connected implies negative}
Let $X$ be a smooth, $\mathbb{A}^1$-connected variety over an algebraically closed field of characteristic $0$. Then $X$ has negative logarithmic Kodaira dimension.
\end{corollary}
\begin{remark}
Suppose, $X \in Sm/k$  
and $Z \subseteq X$ is a closed subscheme of codimension $\geq 2$. If $X$ has negative logarithmic Kodaira dimension, then $X \setminus Z$ has also negative logarithmic Kodaira dimension. If $X$ is rationally connected, then $X \setminus Z$ is also  rationally connected \cite[Exercise 3.7.1, Chapter IV]{kollar}. 
One can ask whether the same is true for $\A^1$-connectedness also: that is, if $X$ is $\A^1$-connected, then whether $X \setminus Z$ is also $\A^1$-connected
 \cite[Conjecture 2.18]{asokdoran}. 
\end{remark}
The conclusion to Theorem \ref{uniruledness generality} is true, if we assume that the generic point of $X$ is $\A^1$-ghost homotopic (Definition \ref{ghost homotopy definition}) to another $K(X)$-valued point.

\begin{theorem} \label{generic point ghost homotopy}
Let $X$ be a smooth $k$-variety and $\eta := Spec(K(X)) \to X$ be the generic point. Suppose, $y : Spec(K(X)) \to X$ is any other $K(X)$-valued point such that $y = \eta \in S^2(X)(K(X))$. Then there is a dominant morphism $\A^1_Y \to X$, for some $Y$ essentially smooth scheme over $k$ with $dim(Y) = dim(X) -1$.

\end{theorem}

\begin{proof}

The argument is along the line of the proof \cite[Theorem 2.11]{choudhury2}.
Suppose, $X \in Sm/k$ is a variety with the generic point $\eta$ such that there is $y \in X(F)$ such that 
$$\eta = y \in \sS^2(X)(F) \text{ and } \eta \neq y \in \sS(X)(F), \text{ where } F = K(X).$$
First observe that if there is some $y \in X(F)$ with $y \neq \eta \in X(F)$ and $\eta = y \in \sS(X)(F)$, then there is a dominant morphism $H: \A^1_U \to X$, for some $U \in Sm/k$ such that $H(0,-) \neq H(1,-): U \to X$. Indeed, since the points $\eta \neq y \in X(F)$ and $\eta = y \in \sS(X)(F)$, so there is some $\theta: \A^1_F \to X$ such that $\theta(0) = \eta$ and $\theta(0) \neq \theta (1) \in X(F)$. Since $F$ is the fraction field of $X$ and $X$ is of finite type over $k$, there is an open subscheme $U \subseteq X$ along with a morphism $\tilde{\theta}: \A^1_U \to X$ such that $\tilde{\theta} \circ i = \theta$, where $i: \A^1_F \to \A^1_U$ is given by $(Id, \eta)$. Since, $\theta(0) \neq \theta(1) \in X(F)$, so $\tilde{\theta}(0) \neq \tilde{\theta}(1) \in X(U)$ and as $\text{Im}(\theta)$ contains $\eta$, so $\text{Im}(\tilde{\theta})$ also contains $\eta$. Thus $\tilde{\theta}$ is a dominant morphism. Hence, such $H$ exists.

Now, assume
$$\eta \neq y \in \sS(X)(F) \text{ and } \eta = y \in \sS^2(X)(F),$$ 
so there is some $\theta: \A^1_F \to \sS(X)$ such that $\theta(0) \neq \theta(1) \in \sS(X)(F)$ and $\theta(0) = \eta$. The element $\theta \in \sS(X)(\A^1_F)$ is given by a Nisnevich covering $f: U \to \A^1_F$ along with a morphism $\phi: U \to X$ such that the following diagram commutes:
\[
\xymatrix{
U \ar[r]^\phi \ar[d]_f & X \ar[d]^\rho \\
\A^1_F \ar[r]_\theta & \sS(X)
}
\] 
here $\rho: X \to \sS(X)$ is the canonical epimorphism. Also there is a Nisnevich covering $W \to U \times_{\A^1_F} U$ such that there is a chain of $\A^1$-homotopies $G_1, \dots, G_n \in X(\A^1_W)$ such that (see Definition \ref{ghost homotopy definition})
$$G_1(0) = \phi \circ \text{pr}_1|_W , \  G_n(1)=\phi \circ \text{pr}_2|_W \text{ and } G_i(0) \neq G_i(1): W \to X \text{ for every }i,$$
where $\text{pr}_1, \text{pr}_2: U \times_{\A^1_F} U \to U$ are the projection morphisms.

Suppose, $U_1, \dots, U_r$ are the irreducible components of $U$ (these are also connected components of $U$, since $U$ is a finite type, smooth $F$-scheme). If possible, $\eta \notin \overline{\phi(U)}$ (that is, $\phi$ is not dominant). Since $\theta(0) = \eta \in \sS(X)(F)$ and $f$ is a Nisnevich covering, there is some
$u \in U(F)$
such that 
$$\phi(u) = \eta \in \sS(X)(F).$$
But since $\eta \notin \overline{\phi(U)}$ (we have assumed), so $\phi(u) \neq \eta \in X(F)$.
From the previous discussion, this implies that there is a dominant morphism 
$$\Tilde{\theta}: \A^1_U \to X \text{ such that }\Tilde{\theta}(0) \neq \Tilde{\theta}(1): U \to X.$$
Thus we can assume that there is an irreducible component $U_j$ (say) of $U$ such that $\eta \in \overline{\phi(U_j)}$.

For the first case, suppose that $\eta \in \overline{\phi(U_i)}$, for every $i$. Then there is an irreducible component $W_0$ of $W$ such that 
$$G_1|_{\A^1_{W_0}}(0) \neq G_1|_{\A^1_{W_0}}(1): W_0 \to X$$
Suppose $W_0$ maps to a component $U_p \times_{\A^1_F} U_q$ of $U \times_{\A^1_F} U$, for some $p, q$. Then, since $\text{pr}_1: U \times_{\A^1_F} U \to U$ is \'etale and $\eta \in \overline{\phi(U_p)}$, so the composition of the morphisms
$$W_0 \hookrightarrow W \to U \times_{\A^1_F} U \xrightarrow{\text{pr}_1} U \xrightarrow{\phi} X$$
is dominant. Since $G_1|_{\A^1_{W_0}}(0)=\phi \circ \text{pr}_1|_{W_0}$, therefore $G_1|_{\A^1_{W_0}}$ is also a dominant morphism.

For the remaining case, suppose there is some irreducible component $U_i$ of $U$ such that $\eta \notin \overline{\phi(U_i)}$. There is some irreducible component $W_0$ of $W$ that maps to $U_j \times_{\A^1_F} U_i$, then
$$\phi \circ \text{pr}_1|_{W_0} \neq \phi \circ \text{pr}_2|_{W_0}.$$
So there is some $t$ such that $G_t|_{\A^1_{W_0}}$ satisfies $G_t|_{\A^1_{W_0}}(0) \neq G_t|_{\A^1_{W_0}}(1)$ and $G_t|_{\A^1_{W_0}}(0) = \phi \circ \text{pr}_1|_{W_0}$. Since $W_0 \to U_j \times_{\A^1_F} U_i \xrightarrow{\text{pr}_1} U_j$ is \'etale and $\eta \in \overline{\phi(U_j)}$, so $G_t|_{\A^1_{W_0}}$ is dominant.

This completes the proof.
\end{proof}
\begin{remark}
In fact the same line of proof of Theorem \ref{generic point ghost homotopy} implies that if $X \in Sm/k$ is a surface and there are $x, y \in X(k)$ such that $x \neq y \in \sS(X)(k)$, but $x = y \in \sS^2(X)(k)$. Then there is a dominant morphism $\A^1_C \to X$, for some smooth curve $C$ over $k$.
More generally in case of surfaces, we have the following theorem (this is proved in \cite[Proposition 4.16]{aritra}):
\end{remark}
\begin{theorem} \label{non uniruled case}
Suupose, $Y \in Sm/k$ is an affine surface and $X \subseteq Y$ is an open subscheme.
If the morphism $\sS(X)(U) \to \sS^2(X)(U)$ is not injective, for some $U \in Sm/k$ then there is a dominant morphism $\A^1_C \to X$, for some smooth curve $C$ over $k$.
\end{theorem}
We have the following consequences to Theorem \ref{non uniruled case}, which says that over $\C$ if $X$ is not $\A^1$-uniruled, then $\pi_0^{\A^1}(X) (Spec \ \C)$ is always non-trivial:
\begin{corollary}
Suppose $X \subseteq Y$ is an open subscheme, where $Y$ is a smooth affine surface. 
\begin{enumerate}
\item If $X$ is not $\A^1$-uniruled, then the map $\sS(X) \to \sS^2(X)$ is an isomorphism. Thus 
$$\pi_0^{\A^1}(X) \cong \sS(X) \cong \sL(X) \text{ is }\A^1\text{-invariant}.$$
\item Suppose, $\pi_0^{\A^1}(X)(Spec \ k)$ is trivial. Then one of the following holds (see also \cite[Proposition 2.16]{choudhury2}):
\begin{enumerate}
\item Either $X$ is $\A^1$-uniruled.
\item $\sS(X)(Spec \ k)$ is trivial.

Therefore, if $k= \C$, then $X$ is always $\A^1$-uniruled and also rational.
\end{enumerate}
\end{enumerate}
\end{corollary}
More generally, if $X$ is a quasi-projective surface which is not $\A^1$-uniruled, then $\pi_0^{\A^1}(X)$ and $\sS(X)$ agree over the field valued sections.
\begin{theorem} \label{quasi projective non negative}
    Suppose $X \in Sm/k$ is a quasi-projective surface of non-negative logarithmic Kodaira dimension. Then $\sS(X)(F) \cong \sS^2(X)(F)$, for every finitely generated separable field extension $F/k$.
\end{theorem}
\begin{proof}
The proof follows is similar to the proof of Theorem \ref{generic point ghost homotopy}.
If possible there are $\alpha, \beta \in X(F)$ such that $\alpha \neq \beta \in \sS(X)(F)$ and $\alpha = \beta \in \sS^2(X)(F)$. Then there is some $\theta \in \sS(X)(\A^1_F)$ such that $\theta(0) = \alpha \in \sS(X)(F)$ and $\theta(1)=\beta \in \sS(X)(F)$ (we can assume this). Following the notations in the proof of Theorem \ref{generic point ghost homotopy}, there is $\phi: U \to X$ and a chain of $\A^1$-homotopies $G_1, \dots,G_n: \A^1_W \to X$ such that $G_1(0)=\phi\circ \text{pr}_1|_W$ and $G_n(1)=\phi\circ \text{pr}_2|_W$. Suppose, $U_1, \dots,U_r$ are the irreducible components of $U$. The proof will be completed considering the following cases:

\textbf{Case 1:} Suppose $\alpha, \beta \in \overline{\phi(U_p)}$, for every $p$. There is some irreducible component $W_0$ of $W$ such that $G_1|_{\A^1_{W_0}}(0) \neq G_1|_{\A^1_{W_0}}(1): W \to X$. Since $F$ is an infinite field ($k$ is algebraically closed) and $W_0$ is finite type $F$-scheme, the $F$-points are dense in $W_0$. So the morphisms $G_1|_{\A^1_{W_0}}(0)$ and $G_1|_{\A^1_{W_0}}(1)$ are different at some $F$-point (say $w$). So the composition of morphisms
$$\gamma: \A^1_F \xrightarrow{(Id, w)} \A^1_{W_0} \xrightarrow{G_1|_{\A^1_{W_0}}} X$$
is non-constant, call it a line $\gamma$ in $X$.
If $\text{Im}(\gamma)$ is closed, then $G_1|_{\A^1_{W_0}}$ must be dominant, since $\overline{G_1|_{\A^1_{W_0}}(0)} = \overline{\phi(U_p)}$, for some $p$ and both $\alpha, \beta$ can not be in $\text{Im}(\gamma)$. 
		
If $\text{Im}(\gamma)$ is not closed, then $\gamma$ can be extended to $\bar{\gamma}: \P^1_F \to X$. The image $\bar{\gamma}(\P^1_F)$ is closed in $X$. Since $\alpha \neq \beta \in \sS(X)(F)$ and $\alpha, \beta \in \overline{\phi(U_p)}$, for every $p$, so $\text{Im}(\bar{\gamma}) \subsetneq \overline{G_1|_{\A^1_{W_0}}}$. So $G_1|_{\A^1_{W_0}}$ is dominant. This implies that $X_F:= X \times_k Spec \ F$ has negative logarithmic Kodaira dimension. But the logarithmic Kodaira dimension is invariant under base extension (\cite[Section 1.1]{invariant}), so from the hypothesis $X_F$ has non-negative logarithmic Kodaira dimension. This is a contradiction.

\


    \textbf{Case 2:} \textbf{Suppose there is some $p$ such that $\alpha \notin \overline{\phi(U_p)}$.}

Case 2 consists two subcases:

\textbf{Subcase 2(a):} \textbf{Renumbering if necessary, assume that there is some $\mathbf{t < n}$ such that 
$$\alpha, \beta \in  \overline{\phi(U_1)}, \dots, \overline{\phi(U_t)}$$ 
and $\alpha \notin \overline{\phi(U_p)}$, for some $p >t$.
}
There is an irreducible component $W_0$ of $W$ that maps to $U_t \times_{\A^1_F} U_p$. Proceeding as before, the homotopy $G_1|_{\A^1_{W_0}}$ satisfies $G_1|_{\A^1_{W_0}}(0) \neq G_1|_{\A^1_{W_0}}(1)$ and $G_1|_{\A^1_{W_0}}$ is dominant. This contradicts the fact that $X$ (or $X_F$) has negative logarithmic Kodaira dimension.

\  
\textbf{Subcase 2(b):} \textbf{There exist $p,q$ such that 
$$\alpha \in \overline{\phi(U_p)}, \ \beta \in \overline{\phi(U_q)} \text{ and } \beta \notin \overline{\phi(U_p)}.$$
}

There is an irreducible component $W_0$ of $W$ that maps to $U_p \times_{\mathbb{A}^1_F} U_q$. Since the projection maps and the morphism $W \to U \times_{\mathbb{A}^1_F} U$ are the \'etale maps, so 
$$\overline{\text{Im}(\phi \circ \text{pr}_1|_{W_0})} = \overline{\phi(U_p)} \text{ and } \overline{\text{Im}(\phi \circ \text{pr}_2|_{V_0})} = \overline{\phi(U_q)}.$$
Since $\beta \notin \overline{\phi(U_p)}$, so
$$\phi \circ \text{pr}_1|_{W_0} \neq \phi \circ \text{pr}_2|_{W_0}.$$
We can assume all the $\mathbb{A}^1$-homotopies $G_1|_{\mathbb{A}^1_{W_0}}, \dots, G_n|_{\mathbb{A}^1_{W_0}}$ satisfy 
$$G_i|_{\mathbb{A}^1_{W_0}}(0) \neq G_i|_{\mathbb{A}^1_{W_0}}(1), \text{ for every }i.$$
So $\text{Im}(G_i|_{\mathbb{A}^1_{V_0}})$ contains a line, given by $\gamma_i: \A^1_F \to X$. 
Either $\text{Im}(\gamma_i)$ is closed in $X$ or $\gamma_i$ can be extended to a morphism $\bar{\gamma_i}: \P^1_F \to X$. Thus if $\overline{\text{Im}(G_i|_{\A^1_{W_0}})}$ is of dimension $1$, for every $i$, then there is a chain consisting $\A^1_F$'s (or $\P^1_F$'s) which contains both $\alpha$ and $\beta$ (since $\alpha \in \overline{\phi(U_p)}$ and $\beta \in \overline{\phi(U_q)}$). This is a contradiction to the fact that $\alpha \neq \beta \in \sS(X)(F)$. So there is some $i$ such that $G_i|_{\mathbb{A}^1_{W_0}}$ is dominant. This contradicts that $X$ (or $X_F$) has negative logarithmic Kodaira dimension.

Therefore, the morphism $\sS(X) \to \sS^2(X)$ induces bijection over $Spec \ F$, for every finitely generated field extension $F/k$.

\end{proof}
\begin{corollary} \
Suppose $X \in Sm/k$ is a quasi-projective surface of non-negative logarithmic Kodaira dimension. Then the map
$$\sS(X) \to \pi_0^{\A^1}(X) $$
induces isomorphism over $Spec \ F$, for every finitely generated separable field extension $F/k$.
\end{corollary}
\begin{remark}
If $X$ is a smooth, proper, non-uniruled $k$-surface, then $\sS(X) \cong \sS^2(X)$ \cite[Corollary 3.14]{bhs}, thus $\pi_0^{\A^1}(X) \cong \sS(X)$. Therefore, following Theorem \ref{quasi projective non negative}, it is natural to ask 
if $X\in Sm/k$ is a quasi-projective surface of non-negative logarithmic Kodaira dimension, then whether $\sS(X) \cong \sS^2(X)$.
\end{remark}

\begin{corollary}
If $X \in Sm/k$ is an $\mathbb{A}^1$-connected surface, where $k$ is an algebraically closed field of characteristic $0$. Then $X$ is rational
(we only use here $\pi_0^{\A^1}(X)(Spec \ k)$ is trivial, see also Theorem \ref{rational implies connected}).
\end{corollary}
\begin{proof}
This is already known by using \cite[Proposition 5.4.2.7]{as} and Theorem \ref{rational implies connected}. We can also prove it using classification of surface and without using Theorem \ref{rationally connected}.

Since $X$ is $\A^1$-connected, so it is $\A^1$-uniruled.
Thus $\bar{X}$ is uniruled, where $\bar{X}$ is a smooth compactification of $X$. So $\bar{X}$ has negative logarithmic Kodaira dimension. Thus from the classification of projective surfaces, either $\bar{X}$ is rational or  
$\bar{X}$ admits a surjective morphism $\pi: \bar{X} \to C$ \cite[Theorem 6.1]{hart}, for some smooth proper curve $C$ such that the general fiber of $\pi$ is $\P^1$. 

If possible, $g(C)>0$. Then $C^\prime := \pi(X)$ is an $\A^1$-rigid curve. But $\pi|_X$ induces a surjective map $\sS^n(X)(k) \to \sS^n(C^\prime)(k)$, for every $n \geq 0$. Since $C^\prime$ is $\A^1$-rigid, so $\sS^n(C^\prime) \cong C^\prime$, for every $n$. The $\A^1$-connectivity of $X$ implies there is some $n$ such that $\sS^n(X)(k)$ is trivial \cite[Remark 6.22]{cb}. This implies that $\sS^n(C^\prime)(k) \cong C^\prime(k)$ is a singleton set. This is a contradiction. Hence $g(C)=0$ and hence $C \cong \P^1$. So $\bar{X}$ is birational to $\P^1 \times \P^1$ \cite[Remark 2.2.1, Chapter V]{hart}.  Hence $\bar{X}$ is rational.
\end{proof}
\begin{remark}
If $X \in Sm/k$ is an $\A^1$-connected, smooth, proper surface, then $\sS(X)(F)$ is trivial, for every finitely generated separable field extension $F/k$ (Theorem \ref{rational implies connected}).
Following \cite[Theorem 4.47]{aritra}, it can be shown that if $X$ is a smooth, affine surface, which is also $\A^1$-connected, then $\sS^2(X)(F)$ is trivial.  However, we do not know such an affine surface $X$ has also trivial
$\sS(X)(F)$ (see \cite[Remark 4.43]{aritra}). 
We also do not know whether an $\A^1$-connected, smooth, quasi-projective surface
$X$
always has trivial $\sS^2(X)(F)$.
\end{remark}
\begin{remark} \label{rationality comparisons}
\begin{enumerate}
\item Theorem \ref{uniruledness generality} says that $\A^1$-connected variety $X$ is $\A^1$-uniruled, that is $X$ admits a dominant morphism $H: \A^1_Y \to X$, for some $Y \in Sm/k$ with $dim(Y)=dim(X)-1$. However, the proof of the theorem does not say anyting about the rationality properties of $Y$. For example $Y$ is unirational, then $X$ would be unirational. 
The classical examples of unirational (in particular rationally connected), non-rational cubic hypersurfaces in $\P^4$ \cite{clemens} are not $\A^1$-connected \cite[Corollary 1.4]{schreieder},
(see also Remark \ref{rational connectedness remark}).
\item Suppose $X$ is a topologically contractible smooth variety over $\C$. If moreover $X$ is a surface, then $X$ is rational (\cite{rationalityI}, \cite{rationalityII}). So the compactification $\bar{X}$ is $\A^1$-connected (Theorem \ref{rational implies connected}). However, there are topologically contractible smooth surfaces over $\C$, which are not $\A^1$-connected \cite[Remark 5.3]{cb}, for example the Ramanujam surface \cite{ra}. It is natural to ask if $X$ is a smooth complex topologically contractible threefold, then whether $\bar{X}$ is $\A^1$-connected or rational (see also \cite[Question 8]{as}).
\end{enumerate}
\end{remark}

\

\section{$\mathbb{A}^1$-homotopy type of log uniruled surfaces and $\mathbb{A}^2$ minus origin}
Let $k$ be an algebraically closed field of characteristic $0$. In this section we describe the $\A^1$-homotopy types of $\A^1$-uniruled affine surfaces. 
The $\A^1$-connected varieties have negative logarithmic Kodaira dimension (Corollary \ref{connected implies negative}). Since the base field is of characteristic $0$, any $\A^2$-form is trivial \cite[Theorem 3]{kamb}. Thus using Miyanishi-Sugie algebraic characterisation of the affine plane \cite[Section 4.1]{miyanishi}, we can characterise the affine plane in $\A^1$-homotopy theory.
\begin{theorem} \cite[Theorem 5.1]{cb} \label{main criteria}
Suppose, $X \in Sm/k$ is an affine surface over an algebraically closed field of characteristic $0$ such that $Pic(X)$ is trivial. Then $X$ is isomorphic to $\A^2_k$ if and only if $X$ is $\A^1$-connected.
\end{theorem}
\begin{theorem}\cite[Theorem 1.1]{cb} \label{criteria char 0}
Let $k$ be a field of characteristic zero and $X \in Sm/k$ is an affine surface. Then $X$ is $\A^1$-contractible if and only if $X$ is isomorphic to $\A^2_k$.
\end{theorem}
If the base scheme is a discrete valuation ring $R$ of equicharacteristic $0$, then using Sathaye's theorem \cite[Theorem 1]{sathaye}, we can characterise the affine plane over a discrete valuation ring.
\begin{corollary} \cite[Theorem 2.21]{choudhury2} \label{dvr}
Let $X$ be a smooth affine scheme over a discrete valuation ring $R$ of equicharacteristic $0$. Suppose, $X$ is of relative dimension $2$. Then $X$ is $\A^1$-contractible if and only if $X$ is isomorphic to $\A^2_R$. 
\end{corollary}
\begin{remark}
\begin{enumerate}
\item If the base field is $\C$, then $\A^1$-contractibility is strictly a stronger notion than the topological contractibility \cite[Corollary 1.3]{cb}. The Ramanujam surface \cite[Section 3]{ra}, the tom-Dieck Petrie surfaces \cite[Theorem A]{dp} are topologically contractible, but not $\A^1$-contractible. Actually they are not even $\A^1$-connected, since they have non-negative logarithmic Kodaira dimension. 

If a variety $X$ is $\A^1$-contractible, then its Voevodsky motive $M(X)$ in $\textbf{DM}_{gm}(\C, \Z)$ is isomorphic to $M(Spec \ \C)$ or $X$ is motivically contractible. A topologically contractible surface over $\C$ is motivically contractible
\cite{asokmotive}.
However, the Ramanujam surface is motivically contractible, but it is not $\A^1$-contractible. It is a natural question whether a topologically contractible threefold $X$ is motivically contractible, 
even under the assumption of $\A^1$-connectedness. One can also ask whether a topologically contractible, smooth, affine threefold, which is also $\A^1$-connected, is even $\A^1$-contractible. It is known that if a topologically contractible threefold $X$ admits a smooth projective compactification $\bar{X}$ such that $\bar{X}$ is rationally connected (in particular $\A^1$-connected), then the Chow groups $CH^i(X)$ are trivial, for $i=1,2,3$ \cite[Theorem 4.0.2]{liu}.
\item In higher dimensions, there are $\A^1$-contractible varieties but it is not isomorphic to the affine space. For example, the Koras-Russell threefold of the first kind is $\A^1$-contractible \cite[Theorem 1.1]{df}, but it not isomorphic to the affine $3$-space (using non-triviality of Makar-Limanov invariant \cite[Theorem 9.9]{fr}). 
Asok-Doran constructed $\A^1$-contractible varieties of dimension $n \geq 4$, which are strictly quasi-affine and they are quotients of affine spaces by $\G_a$-actions \cite[Theorem 5.1, Theorem 5.3]{asok unipotent}. In \cite{higher dimension}, Dubouloz-Ghosh constructed a family of $\A^1$-contractible affine varieties of dimension $n \geq 4$, which are not isomorphic to the affine $n$-space.
\end{enumerate}
\end{remark}
Using the characterisation of the affine plane over the field of characteristic zero, we can describe the following result on the $\A^1$-homotopy type of the codimension $2$ closed complement in the affine plane.
\begin{theorem} \label{codimension 2 complement}
Let $k$ be any field of characteristic $0$. Suppose, $X \in Sm/k$ is an affine surface and $U \subseteq X$ is an open subscheme. Then $U$ is $\A^1$-weakly equivalent to $\A^2_k \setminus \{n \text{-many }k\text{-points }\}$ if and only if $U$ is isomorphic to $\A^2_k \setminus \{n \text{-many }k \text{-points}\}$.
\end{theorem}
\begin{proof}
\cite[Theorem 3.1]{choudhury2}. 
\end{proof}
\begin{remark}
The assumption in Theorem \ref{codimension 2 complement} that $U$ is open in an affine scheme $X$ has been used in the proof. Similar statement to Theorem \ref{codimension 2 complement} in case of quasi-projective surface is not known.
\end{remark}
In higher dimensions, we have the following result.
\begin{theorem} \cite[Theorem 3.4]{choudhury2} \label{dimension 3 poincare}
Suppose, $X$ is the Koras-Russell threefold of the first kind given by
$$X \equiv \{x^mz=y^r+t^s+x\} \subseteq \A^4_k$$ 
and $p=(1,0,1,0) \in X(k)$. Then the quasi-affine threefold $X \setminus \{p\}$ is $\A^1$-weakly equivalent to $\A^3_k \setminus \{(0,0,0)\}$, but $X \setminus \{p\}$ is not isomorphic to $\A^3_k \setminus \{(0,0,0)\}$, as $k$-varieties.
\end{theorem}
\begin{remark} 
Suppose, $X$ is the generalised Koras-Russell variety given by
$$X \equiv \{x_0^mx_1^m\dots x_n^mz=y^r+t^s+x_0\} \subseteq \A^{n+4}_k, \ n \geq 0$$ 
and $p=(1,1, \dots, 1,0,1,0) \in X(k)$. Then the quasi-affine variety $X \setminus \{p\}$ is $\A^1$-weakly equivalent to $\A^{n+3}_k \setminus \{(0,\dots,0)\}$, but $X \setminus \{p\}$ is not isomorphic to $\A^{n+3}_k \setminus \{(0,\dots,0)\}$, as $k$-varieties (\cite[Theorem 6.8]{krishnakumar2}).

\end{remark}
\bibliographystyle{alpha}

\end{document}